\documentclass{amsart}

\usepackage{amsmath}

\usepackage{amsthm}
\usepackage{amsfonts}
\usepackage{dsfont}
\usepackage{enumitem}
\usepackage{color}

\newcommand{\reals}{\mathbb{R}}
\newcommand{\rationals}{\mathbb{Q}}
\newcommand{\complex}{\mathbb{C}}

\newcommand{\integers}{\mathbb{Z}}

\newcommand{\bracketb}[1]{\Big[#1\Big]}

\newcommand{\angles}[1]{\left\langle #1 \right\rangle}

\newcommand{\para}[1]{\left(#1\right)}
\newcommand{\paraa}[1]{\big(#1\big)}
\newcommand{\parab}[1]{\Big(#1\Big)}

\newcommand{\parad}[1]{\Bigg(#1\Bigg)}

\newcommand{\spacearound}[1]{\quad#1\quad}

\renewcommand{\implies}{\spacearound{\Rightarrow}}

\renewcommand{\div}{\operatorname{div}}

\newtheorem{theorem}{Theorem}[section]

\newtheorem{lemma}[theorem]{Lemma}
\newtheorem{proposition}[theorem]{Proposition}

\theoremstyle{definition}
\newtheorem{definition}[theorem]{Definition}
\theoremstyle{remark}

\numberwithin{equation}{section}

\newcommand{\xv}{\vec{x}}
\renewcommand{\mid}{\mathds{1}}
\newcommand{\Ric}{\operatorname{Ric}}
\newcommand{\half}{\frac{1}{2}}
\newcommand{\thalf}{\tfrac{1}{2}}
\newcommand{\Wt}{\tilde{W}}
\newcommand{\Rt}{\tilde{R}}
\newcommand{\exph}{e^{\hbar}}
\newcommand{\expmh}{e^{-\hbar}}

\newcommand{\Fh}{\mathcal{F}_\hbar}
\newcommand{\Fhp}{\mathcal{F}^+_\hbar}
\newcommand{\C}{\mathcal{C}}
\newcommand{\Ch}{\C_{\hbar}}
\newcommand{\Chh}{\widehat{\mathcal{C}}_\hbar}
\renewcommand{\d}{\partial}
\newcommand{\db}{\bar{\d}}
\newcommand{\Der}{\operatorname{Der}}
\newcommand{\du}{\partial_{u}}
\newcommand{\dv}{\partial_{v}}
\renewcommand{\S}{\mathcal{S}}
\newcommand{\XChh}{\mathcal{X}(\Chh)}
\newcommand{\Phib}{\bar{\Phi}}
\newcommand{\g}{\mathfrak{g}}
\newcommand{\Zh}{Z_{\hbar}}
\newcommand{\Zhp}{\Zh^{+}}
\newcommand{\nablad}{\nabla_{\!\d}}
\newcommand{\nabladb}{\nabla_{\!\db}}
\newcommand{\eps}{\varepsilon}
\newcommand{\CcRZ}{C^\infty_0(\reals\times\integers)}

\title{A noncommutative catenoid}
\author{Joakim Arnlind and Christoffer Holm}

\address[Joakim Arnlind]{Dept. of Math.\\
Link\"oping University\\
581 83 Link\"oping\\
Sweden}
\email{joakim.arnlind@liu.se}

\address[Christoffer Holm]{Dept. of Math.\\%
Link\"oping University\\
581 83 Link\"oping\\
Sweden}
\email{chrho686@student.liu.se}

\subjclass[2000]{}

\begin{document}

\maketitle

\begin{abstract}
  A noncommutative algebra corresponding to the classical catenoid is
  introduced together with a differential calculus of derivations. We
  prove that there exists a unique metric and torsion-free connection
  that is compatible with the complex structure, and the curvature is
  explicitly calculated. A noncommutative analogue of the fact that
  the catenoid is a minimal surface is studied by constructing a
  Laplace operator from the connection and showing that the embedding
  coordinates are harmonic. Furthermore, an integral is defined and the
  total curvature is computed.  Finally, classes of left and right
  modules are introduced together with constant curvature connections,
  and bimodule compatibility conditions are discussed in detail.
\end{abstract}

\section{Introduction}

\noindent
In recent years, there has been great progress in understanding
Riemannian aspects of noncommutative geometry and its relation to
topology. For instance, the scalar curvature defined via the heat
kernel has been computed for noncommutative tori and a noncommutative
version of the Gauss-Bonnet theorem has been established (see e.g
\cite{ct:gaussBonnet,fk:gaussBonnet,fk:scalarCurvature,cm:modularCurvature}). In
parallel, one has investigated the role of the Levi-Civita connection
and the curvature tensor, in order to understand to what extent
classical geometrical concepts remain relevant in noncommutative
geometry (see e.g
\cite{dvm:central.bimodules,dvmmm:onCurvature,ac:ncgravitysolutions,bm:starCompatibleConnections,r:leviCivita,a:curvatureGeometric,aw:CGB.sphere,aw:curvature.three.sphere}). In
contrast to the approach via the heat kernel, much of this work has
not been carried out in the setting of $C^\ast$-algebras and spectral
triples, but rather taking a less analytical, and more algebraical,
point of view, constructing curvature through a ``bottom-up'' approach
starting from a hermitian form on a (projective) module. In the
future, it will be interesting to see how these different approaches
may be reconciled. Even though a lot of progress has been made it is
not completely clear what kind of assumptions that are needed in order
to find a unique Levi-Civita connection and what kind of properties
(e.g. symmetries) that one should expect. Therefore, it is useful to
consider particular examples to understand what one might (might not)
expect in the general case.

Another motivation comes from the theory of minimal
surfaces. Whereas the classical theory is by know well developed
(although many interesting questions are still open), its
noncommutative analogue is in an early stage. Several authors have
approached noncommutative minimal submanifolds from different
perspectives (see e.g
\cite{mr:noncommutative.sigma.model,dll:sigma.model.solitions,ach:noncommutative.minimal.surfaces})
but a general framework is still missing. The catenoid is one of the
most well-known minimal surfaces in Euclidean space, and it is
interesting to understand how its properties manifest themselves in
noncommutative geometry. Note that related quantum catenoids have been
considered, although not primarily from a geometrical point of view
\cite{ah:quantizedMinimal,ach:noncommutative.minimal.surfaces}.

In this note we shall construct a noncommutative algebra $\Chh$ that
is closely related to the classical catenoid, which is a (noncompact)
minimal surface embedded in $\reals^3$. The algebra $\Chh$ is not a
$C^\ast$-algebra in any natural way, and typical representations are
given by unbounded operators. However, the algebraic structure is quite
appealing and in many ways similar to the noncommutative torus,
a fact we shall employ to find several natural constructions.

The paper is organized as follows: In Section~\ref{sec:catalgebra} the
algebra $\Chh$, together with a set of derivations, is introduced and
a few basic properties are established. Section~\ref{sec:curvature}
introduces a natural module of vector fields and proves that given a
metric there exists a unique torsion-free connection which is
compatible with the metric and the complex structure. Finally,
Section~\ref{sec:integration} introduces an integral and computes the
total curvature of the metric, and Section~\ref{sec:bimodules}
studies bimodules together with constant curvature connections.

\section{The catenoid algebra}\label{sec:catalgebra}

\noindent
In this section we start from a parametrization of the classical
catenoid and use the Weyl algebra in order to find a natural
definition of a noncommutative catenoid. A parametrization of the
catenoid embedded in $\reals^3$ is given by
\begin{align*}
  \xv(u,v) = \paraa{x^1(u,v),x^2(u,v),x^3(u,v)} = \paraa{\cosh(u)\cos(v),\cosh(u)\sin(v),u}
\end{align*}
for $-\infty<u<\infty$ and $0\leq v\leq 2\pi$. The algebra generated
by the functions $x^1,x^2,x^3$ can in principle also be generated by
$u$, $e^{\pm u}$ and $e^{\pm iv}$. Now, starting from the Weyl
algebra, consisting of two hermitian generators $U$ and $V$ satisfying
\begin{align*}
  [U,V]=i\hbar \mid,
\end{align*}
we shall construct an algebra generated by $U$, $R$ and $W$,
corresponding (formally) to $U$, $e^U$ and $e^{iV}$
respectively. Guided by the Baker-Campbell-Hausdorff formula, giving
e.g.
\begin{align*}
  RW = e^{U}e^{iV}=e^{U+iV+\half[U,iV]}=e^{iV+U+\half[iV,U]-[iV,U]}
  =e^{-\hbar}e^{iV}e^U=e^{-\hbar}WR,
\end{align*}
as well as the formal expansions of $e^U$ and $e^{iV}$ as power series, one
introduces the following relations
\begin{alignat}{2}
  & \Wt W = \mid & &W\Wt=\mid\label{eq:relWWt}\\
  &R\Rt = \mid & &\Rt R =\mid\\
  &RU=UR & &\Rt U=U\Rt\\
  &WR=\exph RW & &W\Rt=\expmh\Rt W\\
  &\Wt R = \expmh R\Wt & &\Wt\Rt=\exph\Rt\Wt\\
  &WU=UW+\hbar W &\qquad &\Wt U=U\Wt-\hbar\Wt\label{eq:relWU}
\end{alignat}
where $\Rt$ and $\Wt$ have been introduced, representing the
inverses of $R$ and $W$. 

\begin{definition}
  Let $\complex\langle U,R,\Rt,W,\Wt\rangle$ be the free associative
  unital algebra on the letters $U,R,\Rt,W,\Wt$ and let $I_\hbar$ be
  the two-sided ideal generated by relations
  \eqref{eq:relWWt}--\eqref{eq:relWU}. We define $\Ch$ as the
  quotient algebra
  \begin{align*}
    \Ch=\complex\langle U,R,\Rt,W,\Wt\rangle\slash I_\hbar.
  \end{align*}
\end{definition}

\noindent
Next, we note that the relations \eqref{eq:relWWt}--~\eqref{eq:relWU}
allows one to always order any element lexicographically (with respect
to the alphabet $U,R,\Rt,W,\Wt$, up to terms of lower total order) and,
moreover, we prove that ordered monomials are linearly independent.

\begin{proposition}
  A basis for $\Ch$ is given by
  \begin{align*}
    e^{\alpha j k} = U^\alpha R^j W^k
  \end{align*}
  for $\alpha\in\integers_{\geq 0}$ and $j,k\in\integers$, where
  $R^{-j}=\Rt^j$ and $W^{-k}=\Wt^k$.
\end{proposition}

\begin{proof}
  In the proof, we shall use the terminology of
  the Diamond lemma \cite{b:diamondlemma} in order to show that
  $\{e^{\alpha jk}\}$ provides a basis for $\Ch$. To this end we start by
  formulating relations \eqref{eq:relWWt}--~\eqref{eq:relWU} in the
  form of a reduction system:
  \begin{alignat*}{3}
    &\sigma_1=(\Wt W,\mid) & &\sigma_2=(W\Wt,\mid) &
    &\sigma_3=(\Rt R,\mid)\\
    &\sigma_4=(R\Rt,\mid) & &\sigma_5=(RU,UR) & &\sigma_6=(\Rt
    U,U\Rt)\\
    &\sigma_7=(WR,\exph RW) & &\sigma_8=(W\Rt,\expmh\Rt W) & 
    &\sigma_9=(\Wt R,\expmh R\Wt)\\
    &\sigma_{10}=(\Wt\Rt,\exph\Rt\Wt)\quad & &\sigma_{11}=(WU,UW+\hbar W)\quad & 
    &\sigma_{12}=(\Wt U,U\Wt-\hbar\Wt),
  \end{alignat*}
  for which we will use the notation $\sigma_i=(W_i,f_i)$.  Let us set
  up a semi-group partial ordering on monomials in $U,R,\Rt,W,\Wt$,
  which is compatible with the above reduction system. That is, a
  semi-group partial ordering such that if $\sigma_i=(W_i,f_i)$ then
  $W_i$ is strictly greater than every monomial in $f_i$. Let
  $p=X^1\cdots X^n$ and $q=Y^1\cdots Y^m$ be two monomials with
  $X^i,Y^i\in\{U,R,\Rt,W,\Wt\}$. We say that $p<q$ if $n<m$ or if
  $n=m$ and $p$ precedes $q$ in lexicographic order with respect to
  the alphabet $U,R,\Rt,W,\Wt$. It is easy to see that this does
  indeed define a semi-group partial ordering compatible with the
  above reduction system. Moreover, it is straightforward to check
  that the partial ordering satisfies the descending chain
  condition. Theorem 1.2 in \cite{b:diamondlemma} tells us that if all
  ambiguities in the above reduction system are resolvable, then a
  basis of $\Ch$ is given by the irreducible monomials. That is,
  monomials which can not be reduced further by using the reduction
  system, replacing $W_i$ by $f_i$ for $i=1,\ldots,12$. How do the
  irreducible monomials look like? It is clear that the monomial
  $U^\alpha R^j W^k$ is irreducible since it does not contain any of
  $W_1,\ldots,W_{12}$. Moreover, there are no other irreducible
  polynomials since one may always use the reduction system to put
  monomials in lexicographic ordering, as well as replacing $R\Rt$,
  $\Rt R$, $W\Wt$ and $\Wt W$ by $\mid$. Thus, it remains to prove
  that all ambiguities are resolvable. There are 20 ambiguities to be
  resolved:
  \begin{align*}
    &(\Wt W)\Wt=\Wt(W\Wt),
    (\Wt W)R=\Wt(WR),
    (\Wt W)\Rt=\Wt(W\Rt),\\
    &(\Wt W)U=\Wt(WU),
    (W\Wt)R=W(\Wt R),
    (W\Wt)\Rt=W(\Wt\Rt),\\
    &(W\Wt)W=W(\Wt W),
    (W\Wt)U=W(\Wt U),
    (\Rt R)\Rt=\Rt(R\Rt),\\
    &(\Rt R)U=\Rt(RU),
    (R\Rt)U=R(\Rt U),
    (R\Rt)R=R(\Rt R),\\
    &(WR)\Rt=W(R\Rt),
    (WR)U=W(RU),
    (W\Rt)R=W(\Rt R),\\
    &(W\Rt)U=W(\Rt U),
    (\Wt R)\Rt=\Wt(R\Rt),
    (\Wt RU)=\Wt(RU),\\
    &(\Wt\Rt)R=\Wt(\Rt R),
    (\Wt\Rt)U=\Wt(\Rt U),
  \end{align*}
  where the parenthesis mark which part of the monomial that is to be
  replaced by using the reduction system. It is straightforward to
  check that they are all resolvable, but let illustrate the procedure by
  explicitly checking $(\Wt W)R=\Wt(WR)$:
  \begin{align*}
    (\mid)R-\Wt(\exph RW)=R-\exph\Wt R W=R-\exph\expmh R\Wt W=R-R=0.
  \end{align*}
  As previously stated, after showing that all ambiguities are
  resolvable, Theorem 1.2 in \cite{b:diamondlemma} implies that the
  monomials $U^\alpha R^jW^k$ provide a basis for $\Ch$. 
\end{proof}

\noindent
We can make $\Ch$ into a $\ast$-algebra by setting 
\begin{alignat*}{3}
  &U^\ast = U &\qquad &R^\ast=R &\qquad &\Rt^\ast=\Rt\\
  &W^\ast=\Wt & &\Wt^\ast = W & &
\end{alignat*}
and noting that the set of relations
\eqref{eq:relWWt}--\eqref{eq:relWU} is invariant with respect to this
involution. From now on we will use the more convenient notation
$R^{-1}=\Rt$ and $W^{-1}=\Wt$. The next results gives a differential
calculus on $\Ch$, in direct analogy with the classical derivatives.

\begin{proposition}\label{prop:derivations}
  There exist hermitian derivations $\du,\dv\in\Der(\Ch)$ such that
  \begin{alignat*}{3}
    &\du U = \mid &\qquad &\du R=R &\qquad &\du W = 0\\
    &\dv U = 0 &\qquad &\dv R=0 &\qquad &\dv W = iW    
  \end{alignat*}
  and $[\du,\dv]=0$.
\end{proposition}

\begin{proof}
  Since derivations are linear and satisfies the product rule, the
  relations in Proposition~\ref{prop:derivations} completely determine
  the action of $\d_u,\d_v$ on $\Ch$. However, in order to be well
  defined, one needs to check that the derivations respect the
  relations in $\Ch$. Thus, one need to check that they are consistent
  with \eqref{eq:relWWt}--~\eqref{eq:relWU}. For instance,
  \begin{align*}
    \d_u(WR-\exph RW)&=\d_u(W)R+W\d_u(R)-\exph\d_u(R)W-\exph R\d_u(W)\\
    &=WR-\exph RW = 0.
  \end{align*}
  In the same way, one may check that $\d_u$ and $\d_v$ respect all
  the relations in the algebra. Moreover, one readily checks that the
  derivations are hermitian; for instance,
  \begin{align*}
    \d_vW^\ast =
    \d_vW^{-1}=-W^{-1}\d_v(W)W^{-1}=-iW^{-1}=-iW^\ast=(iW)^\ast
    =(\d_v W)^\ast,
  \end{align*}
  and analogous computations yield similar results for the remaining
  relations.
\end{proof}

\noindent
Let $\g$ denote the (abelian) complex Lie algebra
generated by $\du$ and $\dv$, and introduce
\begin{align*}
  &\d = \thalf(\du-i\dv)\\
  &\db = \thalf(\du+i\dv).
\end{align*}
For easy reference, let us write out
\begin{alignat*}{3}
  &\d R = \thalf R &\quad &\d R^{-1}=-\thalf R^{-1} &\quad &\d U = \thalf\mid\\
  &\db R = \thalf R & &\db R^{-1}=-\thalf R^{-1} & &\db U = \thalf\mid\\
  &\d W = \thalf W & &\d W^{-1} = -\thalf W^{-1} & &\\
  &\db W = -\thalf W & &\db W^{-1} = \thalf W^{-1}. & &
\end{alignat*}
In the classical setting, functions composed from
$u,e^{\pm u},e^{\pm iv}$ make up a small subset of the smooth
functions on the catenoid; for instance, even though $1+u^2$ is
strictly positive for all $u\in\reals$, there is no element
corresponding to $1/(1+u^2)$ in the algebra. Thus, one would like to
extend $\Ch$ to include elements that correspond to more general
functions. In this paper, our main concern is not to find the most
general algebra for this purpose, but rather to take an opposite
approach, where a minimal extension of $\Ch$ is considered in order to
develop the framework (although, it would be interesting to see how
the results of \cite{w:nuclear.weyl} apply to the current
situation). We will take an approach built on localization, using the
Ore condition. Therefore, one starts by understanding the set of zero
divisors.

\begin{proposition}\label{prop:zero.divisors}
  The algebra $\Ch$ has no zero-divisors.
\end{proposition}

\begin{proof}
  For every $a\in\Ch$ one may write
  \begin{align*}
    a = \sum_{\alpha,j,k}a_{\alpha jk}U^\alpha R^j W^k\qquad
  \end{align*} 
  with $a_{\alpha j k}\in\complex$, and we define the following integers
  \begin{align*}
    N(a)=&\max\{\alpha:\exists j,k\text{ such that }a_{\alpha j k}\neq 0\}\\
    J(a)=&\max\{j:\exists k\text{ such that }a_{N(a) j k}\neq 0\}\\
    K(a)=&\max\{k:a_{N(a)J(a)k}\neq 0\}.
  \end{align*}
  Now, assume that $ab=0$ with $a\neq 0$ and $b\neq 0$. From the relations
  \eqref{eq:relWWt}--~\eqref{eq:relWU} it follows that in the product
  $ab$, there is exactly one term proportional to
  \begin{align*}
    U^{N(a)+N(b)}R^{J(a)+J(b)}W^{K(a)+K(b)},
  \end{align*}
  and the coefficient is given by
  $a_{N(a)J(a)K(a)}b_{N(b)J(b)K(b)}$. Now, since $\{U^\alpha R^jW^k\}$
  is a basis for $\Ch$ it follows that either $a_{N(a)J(a)K(a)}=0$ or
  $b_{N(b)J(b)K(b)}= 0$. However, this contradicts the
  assumption. Hence, if $ab=0$ then at least one of $a$ and $b$ has to
  be zero.
\end{proof}

\noindent
Next, we establish the Ore condition for $\Ch$, which gives a condition for a
non-trivial localization to exist.

\begin{lemma}\label{lemma:ore.condition}
  For every $a,b\in\Ch$, there exists $p,q\in\Ch$ such that
  \begin{align*}
    ap = bq,
  \end{align*}
  and at least one of $p$ and $q$ is non-zero.
\end{lemma}

\begin{proof}
  The proof is a simple argument counting the number of equations and
  the number of variables in a set of linear equations (compare with
  the proof of a similar statement in the case of the Weyl algebra
  \cite{l:classalgebras}). Let us assume that we are given
  $a,b\in\Ch$, and define $N$ to be the an integer such that
  \begin{align*}
    a_{\alpha jk}=b_{\alpha jk}=0
  \end{align*}
  whenever at least one of $\alpha,|j|,|k|$ is greater than $N$. Now,
  we have to find
  \begin{align*}
    p = \sum_{\alpha jk}p_{\alpha jk}U^\alpha R^j W^k\text{ and }
    q = \sum_{\alpha jk}q_{\alpha jk}U^\alpha R^j W^k
  \end{align*}
  such that $ap-bq=0$. Let us choose $p$ and $q$ such that $p_{\alpha
    jk}= q_{\alpha jk}=0$ whenever $\alpha$, $|j|$ or $|k|$ is greater
  than $M$. This implies that $p$ and $q$ together has
  $2M(2M+1)^2$ coefficients to be determined. On the other hand, the
  equation $ap-bq=0$ gives rise to at most $(N+M)(2N+2M+1)^2$ linear equations
  in the coefficients $p_{\alpha jk}$ and $q_{\alpha jk}$, by looking
  at each basis element separately. Choosing $M=4N$ gives
  \begin{align*}
    \#\{\text{variables}\}-\#\{\text{equations}\} = 12N^3+56N^2+15N+1,
  \end{align*}
  which is $\geq 1$ for all $N\geq 0$. Since any homogeneous linear
  system, where the number of variables is strictly greater than the
  number of equations, has a non-zero solution, we conclude that there
  exists a solution where at least one of $p$ and $q$ is non-zero.
\end{proof}

\noindent
Together with Proposition~\ref{prop:zero.divisors},
Lemma~\ref{lemma:ore.condition} implies that there exists a field of
fractions $\Fh$ for $\Ch$ and, furthermore, that the inclusion map
$\lambda:\Ch\to\Fh$ is injective (see e.g. \cite{c:skewFields}). The
algebra $\Fh$ is not particularly well suited as an analogue of the
algebra of functions on the catenoid, since it contains many functions
that are not well-defined at every point of the catenoid. Therefore,
we will construct an extension of $\Ch$ including inverses for a
large class of polynomials.

Let $\Zh(U,R)$ denote the commutative subalgebra of $\Ch$ generated by
$\mid$, $U$, $R$, $R^{-1}$, and define a homomorphism (of commutative
algebras) $\phi:\Zh(U,R)\to C^{\infty}(\reals)$ via
\begin{align*}
  \phi(\mid)=1\qquad  \phi(U) = u\qquad\phi(R)=e^{u}\qquad \phi(R^{-1})=e^{-u}.
\end{align*}
Define the following subset of $\Zh(U,R)$:
\begin{align*}
  \Zhp(U,R) = \{p\in\Zh(U,R):|\phi(p)(u)|>0\text{ for all }u\in\reals\}.
\end{align*}

\begin{lemma}
  $\Zhp(U,R)$ is a multiplicative set.
\end{lemma}

\begin{proof}
  Let $p,q\in\Zhp$, which implies that $|\phi(p)|>0$ and
  $|\phi(q)|>0$. Since $\phi$ is a homomorphism, it follows that
  $|\phi(pq)| = |\phi(p)||\phi(q)|>0$. Hence $pq\in\Zhp(U,R)$.
\end{proof}

\noindent In the following we would like to construct an algebra where
all elements of $\Zhp(U,R)$ are invertible. To this end, let us recall
a few basic results concerning noncommutative localization.  Let
$R,R'$ be rings, and let $\S$ be a subset of $R$. A homomorphism
$f:R\to R'$ is called $\S$-inverting if $f(s)$ is invertible for all
$s\in\S$. A general construction gives the following result:

\begin{proposition}[\cite{c:skewFields}, Proposition 1.3.1]
  Given a ring $R$ and a subset $\S\subseteq R$ there exists a ring
  $R_\S$ and an $\S$-inverting homomorphism $\iota:R\to R_S$ such that
  for every $\S$-inverting homomorphism $f:R\to R'$ there exists a
  unique homomorphism $g:R_S\to R'$ such that $f=g\circ\iota$.
\end{proposition}

\noindent
However, the result does not provide any information on the kernel of
$\iota$, which might be all of $R$. Thus, to guarantee a non-trivial
localization one has to go one step further. First, let us consider
the case when $R=\Ch$ and $\S=\Ch\backslash\{0\}$. Since $\Ch$
satisfies the Ore condition and has no zero-divisors, one may conclude
(cp. \cite[Theorem 1.3.2]{c:skewFields}) that the inclusion map
$\iota_1:\Ch\to\Fh=(\Ch)_\S$ is injective. Next, we let $\S=\Zhp(U,R)$
and consider $\Chh=(\Ch)_{\Zhp(U,R)}$ together with
$\iota:\Ch\to\Chh$. Now, consider the universal property of $\Chh$
applied to $R'=\Fh$ and $\iota_1:\Ch\to\Fh$. Clearly, $\iota_1$ is a
$\Zhp(U,R)$-inverting map from $\Ch$ to $\Fh$. Hence, there exists a
unique homomorphism $g:\Chh\to\Fh$ such that
$\iota_1=g\circ\iota$. Since $\iota_1$ is injective it follows that
$\iota$ is injective. In particular, this implies that the
localization $\Chh$ is non-trivial. Let us summarize the discussion in
the following result.

\begin{proposition}
  There exists an algebra $\Chh$ together with an injective
  $\Zhp(U,R)$-inverting homomorphism $\iota:\Ch\to\Chh$ such that for
  every ring $R'$ and every $\Zhp(U,R)$-inverting homomorphism
  $f:\Ch\to R'$, there exists a unique homomorphism $g:\Chh\to R'$ such
  that $f=g\circ\iota$.
\end{proposition} 

\noindent
For the algebra $\Ch$, a basis was given by monomials of the form
$U^\alpha R^jW^k$. For $\Chh$ one may obtain a corresponding normal form,
using the following result.

\begin{lemma}\label{lemma:WpSwap}
  For every $p\in\Zhp(U,R)$ there exists $q\in\Zhp(U,R)$ such that
  \begin{align*}
    Wp=qW
  \end{align*}
  where $q(U,R)=p(U+\hbar\mid,e^{\hbar}R)$.
\end{lemma}

\begin{proof}
  Assume that $p\in\Zhp(U,R)$ and write
  \begin{align*}
    p = \sum p_{\alpha j}U^\alpha R^j
  \end{align*}
  which gives
  \begin{align*}
    Wp &= \sum p_{\alpha j}WU^\alpha R^j
    =\sum p_{\alpha j}(U+\hbar\mid)^\alpha W R^j\\
    &=\sum p_{\alpha j}(U+\hbar\mid)^\alpha(\exph R)^j W
      =p(U+\hbar\mid,e^hR)W.
  \end{align*}
  Now, let us argue that $p(U+\hbar\mid,\exph R)\in\Zhp(U,R)$. By
  construction, $\phi(p)(u)=p(u,e^u)$ which implies that
  \begin{align*}
    \phi\paraa{p(U+\hbar\mid,\exph R)}(u)=p(u+\hbar,\exph e^u)
    =p(u+\hbar,e^{u+\hbar})=\phi(p)(u+\hbar).
  \end{align*}
  Since $|\phi(p)(u)|>0$ for all $u\in\reals$, it follows that
  $|\phi(p)(u+\hbar)|>0$ for all $u\in\reals$, which shows that 
  $p(U+\hbar\mid,\exph R)\in\Zhp(U,R)$.
\end{proof}

\noindent 
From Lemma~\ref{lemma:WpSwap} one can derive
\begin{align*}
  &Wp^{-1}=p(U+\hbar\mid,\exph R)^{-1}W\\
  &W^{-1}p^{-1}=p(U-\hbar\mid,\expmh R)^{-1}W^{-1}
\end{align*}
for $p\in\Zhp(U,R)$. Thus, using these relations, an element
$a\in\Chh$ can always be written as
\begin{align*}
  a = \sum_{k\in\integers}a_kW^k
\end{align*}
with $a_k\in\Fhp(U,R)$, the commutative subalgebra of $\Chh$ generated
by $\Zh(U,R)$ and inverses of elements in $\Zhp(U,R)$.

\section{Curvature}\label{sec:curvature}

\noindent
In this section we introduce a module of vector fields over $\Chh$,
together with a compatible connection. Given a metric $h$, it turns
out that there exists a unique torsion-free and almost complex
connection that is compatible with $h$. Moreover, the corresponding
curvature tensor is computed as well as the Ricci and scalar
curvature.

For the classical catenoid, parametrized by
\begin{align*}
  \xv(u,v) = \paraa{\cosh(u)\cos(v),\cosh(u)\sin(v),u}
\end{align*}
the space of (complex) vector fields can be spanned by $\phi$ and
$\bar{\phi}$, where
\begin{align*}
  \phi = 2\d\xv = (\sinh(z),-i\cosh(z),1)
\end{align*}
with $z=u+iv$.  Correspondingly, let $\{e_1,e_2,e_3\}$ denote the
canonical basis of the free (right) module $(\Chh)^3$, and set
\begin{align*}
  &\Phi = e_1\Phi^1+e_2\Phi^2+e_3\Phi^3\\
  &\Phib = e_1(\Phi^1)^\ast+e_2(\Phi^2)^\ast+e_3(\Phi^3)^\ast,
\end{align*}
where 
\begin{align*}
  &\Phi^1 = \thalf e^{\half\hbar}(RW-R^{-1}W^{-1})\\
  &\Phi^2 = -\tfrac{i}{2} e^{\half\hbar}(RW+R^{-1}W^{-1})\\
  &\Phi^3 = \mid,
\end{align*}
and let $\XChh$ denote the module generated by $\Phi$ and $\Phib$. Note that
\begin{align*}
  &\Phi^1\sim \frac{1}{2}(e^ue^{iv}-e^{-u}e^{-iv})=\sinh(z)\\
  &\Phi^2\sim -\frac{i}{2}(e^{u}e^{iv}+e^{-u}e^{-iv})=-i\cosh(z)
\end{align*}
when considering the formal correspondence $R\sim e^u$ and $W\sim e^{iv}$ as $\hbar\to 0$.

\begin{proposition}
  $\{\Phi,\Phib\}$ is a basis for $\XChh$, which shows that $\XChh$
  is a free (right) $\Chh$-module of rank 2.
\end{proposition}

\begin{proof}
  Let $a,b\in\Chh$ and assume that $\Phi a+\Phib b = 0$, which is
  equivalent to
  \begin{align*}
    &\Phi^1a+(\Phi^1)^\ast b = 0\\
    &\Phi^2a+(\Phi^2)^\ast b = 0\\
    &\Phi^3a+(\Phi^3)^\ast b = 0
  \end{align*}
  By multiplying these equations from the left by
  $\Phi^1,\Phi^2,\Phi^3$ respectively, and taking their sum, one
  obtains
  \begin{align*}
    \para{\Phi^1(\Phi^1)^\ast+\Phi^2(\Phi^2)^\ast+\Phi^3(\Phi^3)^\ast}b=0
    \implies b=0
  \end{align*}
  since
  \begin{align*}
    &\paraa{\Phi^1}^2+\paraa{\Phi^2}^2+\paraa{\Phi^3}^2=
    \frac{1}{4}e^{\hbar}\parab{(RW-R^{-1}W^{-1})^2-(RW+R^{-1}W^{-1})^2} + \mid\\
                    &=\frac{1}{4}e^{\hbar}\paraa{RWRW-RWR^{-1}W^{-1}-R^{-1}W^{-1}RW+R^{-1}W^{-1}}\\
                    &\quad-\frac{1}{4}e^{\hbar}
                      \paraa{RWRW+R^{-1}W^{-1}R^{-1}W^{-1}+RWR^{-1}W^{-1}+R^{-1}W^{-1}RW}+\mid\\
    &=-\frac{1}{2}e^{\hbar}\paraa{RWR^{-1}W^{-1}+R^{-1}W^{-1}RW}+\mid
      =-\frac{1}{2}e^{\hbar}\paraa{e^{-\hbar}\mid+e^{-\hbar}\mid}+\mid=0.
  \end{align*}
  Analogously, one may multiply the equations from the left with
  $(\Phi^1)^\ast,(\Phi^2)^\ast,(\Phi^3)^\ast$ respectively, and find
  that their sum implies that $a=0$. Since, by definition, $\Phi$ and
  $\Phib$ generate $\XChh$, this shows that $\{\Phi,\Phib\}$ is indeed
  a basis for $\XChh$.
\end{proof}

\noindent
Let $h$ be a hermitian form (or metric) on $\XChh$, i.e.
\begin{align*}
  &h(X,Y+Z) = h(X,Y)+h(X,Z)\\
  &h(X,Y)^\ast = h(Y,X)\\
  &h(X,Ya) = h(X,Y)a
\end{align*}
for all $X,Y,Z\in\XChh$ and $a\in\Chh$. We will assume a diagonal
metric on $\XChh$ given by
\begin{align*}
  h(\Phi,\Phi) = S,\qquad h(\Phib,\Phib)=T,\qquad
  h(\Phi,\Phib) = 0
\end{align*}
with $S$ and $T$ being invertible, implying that the metric is
non-degenerate. For instance, one may consider the induced metric from
the free module by letting $h:(\Chh)^3\times(\Chh)^3\to\Chh$ denote
the bilinear form defined by
\begin{align*}
  h(X,Y) = \sum_{i=1}^3(X^i)^\ast Y^i
\end{align*}
for $X=e_iX^i$ and $Y=e_iY^i$,
for which one computes
\begin{align*}
  &S = h(\Phi,\Phi) = \mid+\thalf \expmh\paraa{R^2+R^{-2}}\\ 
  &T = h(\Phib,\Phib) =  \mid+\thalf \exph\paraa{R^2+R^{-2}}\\
  &h(\Phi,\Phib) = h(\Phib,\Phi)  = 0.
\end{align*}
Note that $S,T\in\Zhp(U,R)$ (implying that they are invertible). We
emphasize that in what follows, $S$ and $T$ are taken to be arbitrary
invertible elements of $\Chh$.

A connection on $\XChh$ is a map $\nabla:\g\times\XChh\to\XChh$ such
that
\begin{align*}
  &\nabla_d(\lambda X+Y) = \lambda\nabla_dX + \nabla_dY\\
  &\nabla_{\lambda d+d'}X = \lambda\nabla_dX + \nabla_{d'}X\\
  &\nabla_{d}(Xa) = \paraa{\nabla_dX}a+Xd(a),
\end{align*}
for $\lambda\in\complex$, $X,Y\in\XChh$, $d,d'\in\g$, and
$a\in\Chh$. A connection is called \emph{hermitian} if
\begin{align*}
  dh(X,Y) = h(\nabla_{d^\ast}X,Y) + h(X,\nabla_dY),
\end{align*}
for $d\in\g$ and $X,Y\in\XChh$. Moreover, we say that $\nabla$ is
\emph{torsion-free} if
\begin{align*}
  \nabla_{\d}\Phib=\nabla_{\db}\Phi.
\end{align*}
Let us introduce an almost complex structure $J:\XChh\to\XChh$ by setting
\begin{align*}
  &J\Phi = i\Phi\\
  &J\Phib = -i\Phib
\end{align*}
and extending $J$ to $\XChh$ as a (right) $\Chh$-module
homomorphism. A connection is called \emph{almost complex} if 
\begin{align*}
  (\nabla_d J)(X)\equiv \nabla_dJ(X)-J\nabla_d X = 0
\end{align*}
for all $d\in\g$ and $X\in\XChh$.

With respect to the basis $\{\Phi,\Phib\}$, a connection on $\XChh$ is
given by choosing arbitrary $\Gamma^{a}_{bc}\in\Chh$ (for
$a,b,c\in\{1,2\}$) and setting
\begin{align*}
  \nabla_a X\equiv\nabla_{\d_a}X = \Phi_b\d_aX^b+\Phi_c\Gamma^{c}_{ab}X^b
\end{align*}
where $\Phi_1=\Phi$, $\Phi_2=\Phib$, $\d_1=\d$, $\d_2=\db$ and
$X=\Phi_aX^a$. Demanding that the connection is torsion-free
immediately gives that $\Gamma^a_{bc}=\Gamma^a_{cb}$.

\begin{lemma}\label{lemma:almost.complex.connection}
  Let $\nabla$ be a connection on $\XChh$ given as
  \begin{align*}
    \nabla_aX = \Phi_b\d_aX^b+\Phi_c\Gamma^{c}_{ab}X^b.
  \end{align*}
  The connection is almost complex if and only if
  \begin{align*}
    &\Gamma^{2}_{11}=\Gamma^{1}_{22}=0\\
    &\Gamma^2_{21}=\Gamma^{1}_{12}=0.
  \end{align*}
\end{lemma}

\begin{proof}
  The condition for $\nabla$ to be almost complex is equivalent to
  \begin{align}
    &\nablad(J\Phi) - J\nablad\Phi = 0\\
    &\nabladb(J\Phi) - J\nabladb\Phi = 0\\
    &\nablad(J\Phib) - J\nablad\Phib = 0\\
    &\nabladb(J\Phib) - J\nabladb\Phib = 0.
  \end{align}
  For an arbitrary connection $\nabla$, given by
  \begin{align*}
  \nabla_aX\equiv\nabla_{\d_a}X = \Phi_b\d_aX^b+\Phi_c\Gamma^{c}_{ab}X^b,
  \end{align*}
  these equations are equivalent to
  \begin{alignat*}{2}
    2i\Phib\Gamma^2_{11} &= 0 &\qquad 2i\Phib\Gamma^{2}_{21}&=0\\
    -2i\Phi\Gamma^1_{12} &= 0 & -2i\Phi\Gamma^1_{22}&=0,
  \end{alignat*}
  which immediately gives the desired result since $\{\Phi,\Phib\}$ is
  a basis for $\XChh$.
\end{proof}

\noindent
Thus, it follows from Lemma~\ref{lemma:almost.complex.connection} that
a connection is torsion-free and almost complex if and only if
\begin{align}\label{eq:torsionfree.acomplex.connection}
  &\nabla_\d\Phi = \Phi\Gamma^1,\quad \nabla_{\db}\Phib = \Phib\Gamma^2,\quad \nabla_{\db}\Phi = \nabla_{\d}\Phib=0
\end{align}
for some $\Gamma^1,\Gamma^2\in\Chh$. If, in addition, the connection
is hermitian the connection coefficients will be uniquely fixed,
as formulated in the next result.

\begin{theorem}\label{thm:levi.civita.connection}
  There exists a unique hermitian torsion-free almost complex
  connection $\nabla$ on $\XChh$, given by
  \begin{align*}
    \nabla_\d\Phi &= \Phi S^{-1}\d S\\ 
    \nabla_{\db}\Phib &= \Phib T^{-1}\db T\\ 
    \nabla_{\db}\Phi &= \nabla_\d\Phib = 0.
  \end{align*}
\end{theorem}

\begin{proof}
  As noted in \eqref{eq:torsionfree.acomplex.connection}, an arbitrary
  torsion-free almost complex connection may be written as
  \begin{align*}
    &\nabla_\d\Phi = \Phi\Gamma^1,\quad \nabla_{\db}\Phib = \Phib\Gamma^2,\quad \nabla_{\db}\Phi = \nabla_{\d}\Phib=0.    
  \end{align*}
  To prove that the connection is compatible with the metric, one
  needs to show that
  \begin{align*}
    &\d h(\Phi_a,\Phi_b) = h(\nabladb\Phi_a,\Phi_b)+h(\Phi_a,\nablad\Phi_b)\\
    &\db h(\Phi_a,\Phi_b) = h(\nablad\Phi_a,\Phi_b)+h(\Phi_a,\nabladb\Phi_b)
  \end{align*}
  for $a,b\in\{1,2\}$. For $a\neq b$, the above conditions are void
  since each term is separately zero. For $a=b=1$ one obtains
  \begin{align*}
    \d S = S\Gamma^1\quad\text{and}\quad
    \db S = (\Gamma^1)^\ast S
  \end{align*}
  both giving $\Gamma^1=S^{-1}\d S$. Similarly, for $a=b=2$ one
  obtains $\Gamma^2=T^{-1}\db T$. Thus, demanding that a connection
  is metric, torsion-free and almost complex uniquely fixes the
  connection components $\Gamma^{a}_{bc}$.
\end{proof}

\noindent
The curvature $R(\d_a,\d_b)X=\nabla_a\nabla_bX-\nabla_b\nabla_aX$ is
easily computed to be
\begin{align}
  &R(\d,\db)\Phi = -\Phi\db\paraa{S^{-1}\d S} \\
  &R(\d,\db)\Phib = \Phib\d\paraa{T^{-1}\db T}
\end{align}
and since $\XChh$ is a free module, one has uniquely defined curvature
components $R(\d_a,\d_b)\Phi_c=\Phi_p{R^p}_{cab}$ given by
\begin{align*}
  &{R^{1}}_{112}=-\db\paraa{S^{-1}\d S}\qquad
  {R^2}_{212}=\d\paraa{T^{-1}\db T}\\
  &{R^{1}}_{212}={R^{2}}_{112}=0.
\end{align*}
One may also proceed to define
$R_{abpq}=h(\Phib_a,R(\d_p,\d_q)\Phi_b)$, where $\Phib_1=\Phib$ and
$\Phib_2=\Phi$, giving
\begin{align*}
  &R_{1212}=T\d\paraa{T^{-1}\db T}\qquad
  R_{2112}=-S\db\paraa{S^{-1}\d S}\\
  &R_{1112}=R_{2212}=0
\end{align*}
and we note that $R_{abpq}$ does not enjoy the all the symmetries of the
classical Riemann tensor due to the fact that $S\neq T$ in the
noncommutative setting.

Furthermore, there are natural definitions of both Ricci and scalar
curvature
\begin{align*}
  &\Ric_{ab} = {R^p}_{apb}\\
  &R = -T^{-1}R_{1212}S^{-1}+S^{-1}R_{2112}T^{-1}
\end{align*}
giving
\begin{align*}
  &\Ric_{12}=-\db\paraa{S^{-1}\d S}\qquad
    \Ric_{21}=-\d\paraa{T^{-1}\db T}\\
  &\Ric_{11}=\Ric_{22} = 0.\\
  &R = -\d\paraa{T^{-1}\db T}S^{-1}-\db\paraa{S^{-1}\d S}T^{-1}
\end{align*}

\noindent
In all these expressions we note the appearance of ``logarithmic''
derivatives, in analogy with $\d\db\ln(f)=\d(f^{-1}\db f)$.  For the
unique connection in Theorem~\ref{thm:levi.civita.connection} one may
define both gradient and divergence in a natural way.
\begin{definition}
  For $f\in\Chh$ and $X=\Phi a+\Phib b\in\XChh$, define
  \begin{align*}
    \nabla f &= \Phi S^{-1}\db f+\Phib T^{-1}\d f\\
    \div(X) &= S^{-1}\d(Sa)+T^{-1}\db(Tb)\\
    \Delta(f) &= \div(\nabla f).
  \end{align*}
  An element $f\in\Chh$ is called \emph{harmonic} if $\Delta(f)=0$.
\end{definition}

\begin{proposition}
  If $f\in\Chh$ then
  \begin{align*}
    \Delta(f) = (S^{-1}+T^{-1})\d\db f.
  \end{align*}
\end{proposition}

\begin{proof}
  The proof consists of a straight-forward computation:
  \begin{align*}
    \div(\nabla f) &= \div\paraa{\Phi S^{-1}\db f + \Phib T^{-1}\d f}
    =S^{-1}\d\paraa{SS^{-1}\db f}+T^{-1}\db\paraa{TT^{-1}\d f}\\
    &= S^{-1}\d\db f + T^{-1}\db\d f = (S^{-1}+T^{-1})\d\db f
  \end{align*}
  since $[\d,\db]=0$.
\end{proof}

\noindent In classical geometry, the fact that the catenoid is a
minimal surface may be characterized by demanding that the embedding
coordinates $x^1,x^2,x^3$ are harmonic. A similar statement holds for
$\Chh$; namely, one notes that if
\begin{align*}
  &X^1 = \thalf e^{\thalf\hbar}\paraa{RW+R^{-1}W^{-1}+(RW+R^{-1}W^{-1})^\ast}\\
  &X^2 = -\tfrac{i}{2}e^{\thalf\hbar}\paraa{RW-R^{-1}W^{-1}-(RW-R^{-1}W^{-1})^\ast}\\
  &X^3 = U
\end{align*}
then $(X^i)^\ast = X^i$ and $\d X^i=\Phi^i$ for $i=1,2,3$, in analogy
with the classical embedding coordinates into $\reals^3$. One may
readily check that the noncommutative embedding coordinates are
harmonic
\begin{align*}
  \Delta(X^1) &= (S^{-1}+T^{-1})\db\d X^1 = (S^{-1}+T^{-1})\db\Phi^1\\
              &= \thalf e^{\thalf\hbar}(S^{-1}+T^{-1})\db(RW-R^{-1}W^{-1})\\
              &= \thalf e^{\thalf\hbar}(S^{-1}+T^{-1})
                \paraa{RW-RW+R^{-1}W^{-1}-R^{-1}W^{-1}} = 0\\
  \Delta(X^2) &= (S^{-1}+T^{-1})\db\d X^2 = (S^{-1}+T^{-1})\db\Phi^2\\
              &= \frac{i}{2}e^{\thalf\hbar}(S^{-1}+T^{-1})\db(RW+R^{-1}W^{-1})\\
              &= \thalf e^{\thalf\hbar}(S^{-1}+T^{-1})
                \paraa{RW-RW-R^{-1}W^{-1}+R^{-1}W^{-1}} = 0\\
  \Delta(X^3) &= (S^{-1}+T^{-1})\db\d X^3 = (S^{-1}+T^{-1})\db(\tfrac{1}{2}\mid) = 0.
\end{align*}

\section{Integration and total curvature}\label{sec:integration}

\noindent
Let us introduce a concept of integration in analogy with integration
on the classical catenoid. The total integral of a function on the
catenoid, with respect to the induced metric can be computed in local
coordinates as
\begin{align*}
  \tau(f) =\int_{-\infty}^\infty\parad{\int_0^{2\pi}f(u,v)\cosh^2(u)dv}du
\end{align*}
whenever the integral exists. For a function, expressible as
\begin{align*}
  f(u,v) = \sum_{k\in\integers}f_k(u,e^u)e^{ikv}
\end{align*}
we note that $\tau(f)=\tau(f_0)$. To define a corresponding
linear functional on $\Chh$, we start by extending the map
$\phi:\Zh(U,R)\to C^\infty(\reals)$ (as defined in Section
\ref{sec:catalgebra}) to $\Fhp(U,R)$ by setting
$\phi(p^{-1})=1/\phi(p)$ for $p\in\Zhp(U,R)$. Then, given $a\in\Ch$
\begin{align*}
  a = \sum_{k\in\integers}a_kW^k
\end{align*}
we set
\begin{align*}
  \tau_0(a) = 2\pi\int_{-\infty}^\infty\phi(a_0)du
\end{align*}
whenever the integral exists. Note that $\tau_0$ is in general not a
trace.

Given a conformal metric of the type
\begin{align*}
  h(\Phi,\Phi) = h(\Phib,\Phib) = S\qquad h(\Phi,\Phib) = 0
\end{align*}
where $S\in\Fhp(U,R)$ is invertible, one introduces
an integral with respect to the corresponding volume form as
\begin{align*}
  \tau_h(a) = 4\pi\int_{-\infty}^\infty\phi(a_0)\phi(S)\,du.
\end{align*}
The extra factor of two is introduced for convenience due to the fact
that for a conformal metric with $h(\d,\d)=h(\db,\db)=s$ one finds
that $g(\d_u,\d_u)=g(\d_v,\d_v)=2s$.  As an illustration of the above
concepts, let us compute the noncommutative total curvature, i.e. the
integral of the Gaussian curvature (defined to be half of the scalar
curvature). From Section~\ref{sec:curvature} one finds the Gaussian
curvature
\begin{align*}
  K = -\frac{1}{4}\d_u\paraa{S^{-1}\d_u S}S^{-1}
\end{align*}
since $\d S = \db S=\thalf\d_u S$ when $S\in\Fhp(U,R)$, and the
corresponding integral gives
\begin{align*}
  \tau_h(K) = -\pi\int_{-\infty}^\infty\d_u\paraa{s^{-1}\d_us}du
  =-\pi\bracketb{s^{-1}\d_u s}_{-\infty}^\infty,
\end{align*}
where $s=\phi(S)$.  For instance, choosing a metric in analogy with
the induced metric from $\reals^3$
\begin{align*}
  &S = \frac{1}{4}(R+R^{-1})^2\implies s=\phi(S)=\cosh^2(u)\implies\\
  &\tau_h(K) = -\pi\bracketb{2\tanh(u)}_{-\infty}^{\infty}=-4\pi,
\end{align*}
in accordance with the classical result.

\section{Bimodules}\label{sec:bimodules}

\noindent
In this section we will introduce classes of left and right modules
over $\Ch$, as well as study compatibility conditions for
bimodules. Moreover, constant curvature connections are introduced and
their relations to the bimodule structure is discussed.

Let $\CcRZ$ denote the space of complex valued smooth functions on
$\reals\times\integers$ with compact support (in both variables), together with the inner
product
\begin{align}\label{eq:module.inner.product}
  \angles{\xi,\eta} = \sum_{k=-\infty}^\infty
  \int_{-\infty}^\infty\xi(x,k)\bar{\eta}(x,k)dx.
\end{align}

\begin{proposition}\label{prop:left.right.module}
  Let $\lambda_0,\lambda_1,\eps\in\reals$ and $r\in\integers$. If
  $\lambda_0\eps+\lambda_1r=-\hbar$ then
  \begin{align*}
    (W\xi)(x,k)&=\xi(x-\eps,k-1)\\
    (W^{-1}\xi)(x,k)&=\xi(x+\eps,k+1)\\
    (R\xi)(x,k) &= e^{\lambda_0x+\lambda_1k}\xi(x,k)\\
    (R^{-1}\xi)(x,k) &= e^{-\lambda_0x-\lambda_1k}\xi(x,k)\\
    (U\xi)(x,k)&=(\lambda_0x+\lambda_1k)\xi(x,k)
  \end{align*}
  for $\xi\in\CcRZ$, defines a
  left $\Ch$-module structure on
  $\CcRZ$. Correspondingly,
  \begin{align*}
    (\xi W)(x,k) &= \xi(x-\eps',k-r')\\
    (\xi W^{-1})(x,k) &= \xi(x+\eps',k+r')\\
    (\xi R)(x,k) &= e^{\mu_0x+\mu_1k}\xi(x,k)\\
    (\xi R^{-1})(x,k) &= e^{-\mu_0x-\mu_1 k}\xi(x,k)\\
    (\xi U)(x,k) &= (\mu_0x+\mu_1k)\xi(x,k)
  \end{align*}
  defines a right $\Ch$-module structure on $\CcRZ$ if
  $\mu_0,\mu_1,\eps'\in\reals$ and $r'\in\integers$ such that
  $\mu_0\eps'+\mu_1r'=\hbar$. Moreover, both the left and right module
  structures are compatible with the inner product; i.e.
  \begin{align*}
    \angles{a\xi,\eta}=\angles{\xi,a^\ast\eta}\quad\text{and}\quad
    \angles{\xi a,\eta} = \angles{\xi,\eta a^\ast}
  \end{align*}
  for $a\in\Ch$ and $\xi,\eta\in\CcRZ$.
\end{proposition}

\begin{proof}
  Let us show that the above definitions define a left module
  structure. The right module structure is checked in an analogous
  way.

  It follows immediately from the definitions that
  \begin{align*}
    &WW^{-1}\xi=\xi,\quad W^{-1}W\xi = \xi,\quad
    RR^{-1}\xi=\xi,\quad R^{-1}R\xi = \xi\\
    &[R,U]\xi=0,\quad[R^{-1},U]\xi = 0.
  \end{align*}
  Thus, it remains to check the following relations:
  \begin{align*}
    &WR=\exph RW\\
    &WU=UW+\hbar W.
  \end{align*}
  One gets
  \begin{align*}
    &WR\xi(x,k) -e^\hbar RW\xi(x,k)=
    We^{\lambda_0x+\lambda_1k}\xi(x,k)
      -e^\hbar R\xi(x-\eps,k-r)\\
    &=e^{\lambda_0(x-\eps)+\lambda_1(k-r)}\xi(x-\eps,k-r)
      -e^{\hbar}e^{\lambda_0x+\lambda_1k}\xi(x-\eps,k-r)\\
    &=e^{\lambda_0x+\lambda_1k}
      \paraa{e^{-\lambda_0\eps-\lambda_1r}-e^{\hbar}}\xi(x-\eps,k-r)=0,
  \end{align*}
  by using that $\lambda_0\eps+\lambda_1r=-\hbar$. Finally, one computes
  \begin{align*}
    WU&\xi(x,k)-UW\xi(x,k)-\hbar W\xi(x,k)\\
      &=W(\lambda_0x+\lambda_1k)\xi(x,k)
        -U\xi(x-\eps,k-r)-\hbar\xi(x-\eps,k-r)\\
      &=\paraa{\lambda_0(x-\eps)+\lambda_1(k-r)}\xi(x-\eps,k-r)
        -(\lambda_0x+\lambda_1k)\xi(x-\eps,k-r)\\
    &\hspace{70mm}-\hbar\xi(x-\eps,k-r)\\
      &=\paraa{-\lambda_0\eps-\lambda_1r-\hbar}\xi(x-\eps,k-r)=0,
  \end{align*}
  again using that $\lambda_0\eps+\lambda_1r=-\hbar$.  It is now
  straightforward to check that the module structure is compatible
  with the inner product; e.g.
  \begin{align*}
    \angles{W\xi,\eta} = \sum_{k=-\infty}^\infty\int_{-\infty}^\infty
    \xi(x-\eps,k-r)\bar{\eta}(x,k)dx
  \end{align*}
  which, by setting $l=k-r$ and $y=x-\eps$, becomes
  \begin{align*}
    \angles{W\xi,\eta} = \sum_{l=-\infty}^\infty\int_{-\infty}^\infty
    \xi(y,l)\bar{\eta}(y+\eps,l+r)dy=\angles{\xi,W^{-1}\eta}
    =\angles{\xi,W^\ast\eta},
  \end{align*}
  and similar computations are carried out for the remaining generators.
\end{proof}

\noindent
In order for $\CcRZ$ to be a $\Ch-\C_{\hbar'}$-bimodule, one has to
demand that the two structures in
Proposition~\ref{prop:left.right.module} are compatible; i.e. that
$A(\xi B)=(A\xi)B$ for all $A\in\Ch$ and $B\in\C_{\hbar'}$. This
induces certain compatibility conditions on the parameters, as
formulated in the next result.

\begin{proposition}\label{prop:bimodule.structure}
  Let $\lambda_0,\lambda_1,\eps,\eps'\in\reals$ and $r,r'\in\integers$
  such that $\lambda_0\eps+\lambda_1r=-\hbar$ and
  $\mu_0\eps'+\mu_1r'=-\hbar'$. If $\lambda_0\eps'+\lambda_1r'=0$ and
  $\mu_0\eps+\mu_1r=0$ then $\CcRZ$ is
  a $\Ch-\mathcal{C}_{\hbar'}$-bimodule.
\end{proposition}

\begin{proof}
  To prove that $\CcRZ$ is a
  bimodule, one has to show that the left and right module structures
  given in Proposition~\ref{prop:left.right.module} are compatible;
  i.e.
  \begin{align*}
    \paraa{(A\xi)B}(x,k) = \paraa{A(\xi B)}(x,k)
  \end{align*}
  for all $A\in\Ch$ and $B\in\C_{\hbar'}$. It is enough to check
  compatibility for the generators, for instance
  \begin{align*}
    \paraa{(W\xi)R}(x,k)&-\paraa{W(\xi R)}(x,k)=
      e^{\mu_0x+\mu_1k}(W\xi)(x,k)-(\xi R)(x-\eps,k-r)\\
    &=\paraa{e^{\mu_0x+\mu_1k}-e^{\mu_0(x-\eps)+\mu_1(k-r)}}\xi(x-\eps,k-r)
      =0
  \end{align*}
  since $\mu_0\eps+\mu_1r = 0$. Similarly, one gets
  \begin{align*}
    \paraa{(R\xi)W}(x,k)&-\paraa{R(\xi W)}(x,k)=
    (R\xi)(x-\eps',k-r')-e^{\lambda_0x+\lambda_1k}(\xi W)(x,k)\\
    &=\paraa{e^{\lambda_0(x-\eps')+\lambda_1(k-r')}-e^{\lambda_0x+\lambda_1k}}\xi(x-\eps',k-r')=0
  \end{align*}
  since $\lambda_0\eps'+\lambda_1r'=0$. The remaining compatibility
  conditions may be checked in an analogous way.
\end{proof}

\noindent
Note that it possible to obtain a $\Ch-\C_{\hbar'}$-bimodule for
arbitrary $\hbar,\hbar'\in\reals$; namely, for $\eps\neq\eps'$
one may set
\begin{align*}
  &r = r' = 1\qquad
    \lambda_0=-\frac{\hbar}{\eps-\eps'}\qquad
  \mu_0=-\frac{\hbar'}{\eps-\eps'}\\
  &\lambda_1 = \frac{\hbar\eps'}{\eps-\eps'}\qquad
    \mu_1 = \frac{\hbar'\eps}{\eps-\eps'},
\end{align*}
fulfilling the requirements of
Proposition~\ref{prop:bimodule.structure}.

Define linear maps
$\nabla_u,\nabla_v:\CcRZ\to\CcRZ$ via
\begin{align}\label{eq:module.connection.def}
  (\nabla_u\xi)(x,k) = \alpha\frac{d\xi}{dx}(x,k)
  \qquad\text{and}\qquad
  (\nabla_v\xi)(x,k) = \beta x\xi(x,k)
\end{align}
for $\alpha,\beta\in\complex$. It is straightforward to check that
\begin{align*}
  &\nabla_u(a\xi) = a\nabla_u\xi + (\d_u a)\xi\\
  &\nabla_v(a\xi) = a\nabla_v\xi + (\d_v a)\xi
\end{align*}
for all $a\in\Ch$ if and only if $\alpha=1/\lambda_0$ and
$\beta=i/\eps$. Similarly, it holds that
\begin{align*}
  &\nabla_u(\xi a) = (\nabla_u\xi)a + \xi(\d_u a)\\
  &\nabla_v(\xi a) = (\nabla_v\xi)a + \xi(\d_v a)
\end{align*}
for all $a\in\C_{\hbar'}$ if and only if $\alpha=1/\mu_0$ and
$\beta=i/\eps'$. Thus, \eqref{eq:module.connection.def} defines a left
resp. right connection on $\CcRZ$
of constant curvature $\alpha\beta$; i.e.
\begin{align*}
  \nabla_u\nabla_v\xi-\nabla_v\nabla_u\xi = \alpha\beta\xi.
\end{align*}
By choosing suitable parameters, one obtains a bimodule
connection.

\begin{proposition}
  Assume that $\CcRZ$ is a
  $\Ch-\mathcal{C}_{\hbar'}$-bimodule as in
  Proposition~\ref{prop:bimodule.structure}, and that 
  \begin{align*}
    &(\nabla_u\xi)(x,k) = \frac{1}{\lambda_0}\frac{d\xi}{dx}(x,k)\\
    &(\nabla_v\xi(x,k)) = \frac{i}{\eps}x\xi(x,k)
  \end{align*}
  is a bimodule connection on $\CcRZ$.
  \begin{enumerate}
    \item If $\hbar=\hbar'$ then $\hbar=\hbar'=0$,
    \item if $\hbar\neq\hbar'$ then $\hbar/\hbar'\in\rationals$ and
      \begin{align*}
        &\lambda_0=\mu_0=\frac{\hbar r'}{\eps(r-r')}\qquad
        \lambda_1=-\frac{\hbar}{r-r'}\qquad
        \mu_1 = -\frac{\hbar'}{r-r'}
      \end{align*}
      for arbitrary $\eps=\eps'\in\reals$ and $r,r'\in\integers$ such
      that $r/r'=\hbar/\hbar'$. Moreover,
      \begin{align*}
        \nabla_u\nabla_v\xi(x,k)-\nabla_v\nabla_u\xi(x,k) =
        i\frac{\hbar-\hbar'}{\hbar\hbar'}\xi(x,k).       
      \end{align*}
  \end{enumerate}
\end{proposition}

\begin{proof}
  For a bimodule connection one must necessarily have
  $\lambda_0=\mu_0$ and $\eps=\eps'$. Together with the conditions in
  Proposition~\ref{prop:bimodule.structure} one obtains the equations
  \begin{align}
    &\lambda_0\eps+\lambda_1r=-\hbar\label{eq:lelr}\\
    &\lambda_0\eps+\mu_1r' = \hbar'\label{eq:lemrp}\\
    &\lambda_0\eps+\lambda_1r'=0\label{eq:lelrp}\\
    &\lambda_0\eps+\mu_1r = 0.\label{eq:lemr}
  \end{align}
  Since $\lambda_0,\eps\neq 0$ (in order for the bimodule connection to
  be defined), equations \eqref{eq:lelrp} and \eqref{eq:lemr} imply
  that $r,r'\neq 0$. Thus, one may solve these equations as
  \begin{align}
    \lambda_1 = -\frac{\lambda_0\eps}{r'}\qquad
    \mu_1 = -\frac{\lambda_0\eps}{r}.\label{eq:lomo.solution}
  \end{align}
  Inserting \eqref{eq:lomo.solution} into \eqref{eq:lelr} and
  \eqref{eq:lemrp} gives
  \begin{align}
    &\lambda_0\eps\parab{1-\frac{r}{r'}} =-\hbar\label{eq:lerrph}\\
    &\lambda_0\eps\parab{1-\frac{r'}{r}} = \hbar'\label{eq:lerrphp}
  \end{align}
  Now, assume that $\hbar=\hbar'$. Summing \eqref{eq:lerrph} and
  \eqref{eq:lerrphp} yields
  \begin{align*}
    \frac{r}{r'}+\frac{r'}{r} = 2,
  \end{align*}
  which has the unique solution $r=r'$, implying that $\hbar=\hbar'=0$
  via \eqref{eq:lerrph}. This proves the first part of the
  statement. Next, assume that $\hbar\neq \hbar'$.

  First we note that neither $\hbar$ nor $\hbar'$ can be zero, since
  that implies (by \eqref{eq:lerrph} and \eqref{eq:lerrphp}) that
  $\hbar=\hbar'=0$ contradicting the assumption that
  $\hbar\neq\hbar'$. Thus, we can assume that $\hbar,\hbar'\neq 0$,
  which implies that $r\neq r'$ (again, by \eqref{eq:lerrph} and
  \eqref{eq:lerrphp}). Solving \eqref{eq:lerrph} for $\lambda_0$ gives
  \begin{align}
    \lambda_0 = \frac{\hbar r'}{\eps(r-r')}\label{eq:lz.solution}
  \end{align}
  which, when inserted in \eqref{eq:lerrphp}, gives
  \begin{align*}
    \frac{r}{r'} = \frac{\hbar}{\hbar'}.
  \end{align*}
  Hence, the quotient $\hbar/\hbar'$ is necessarily rational, since
  $r,r'\in\integers$, and inserting \eqref{eq:lz.solution} into
  \eqref{eq:lomo.solution} yields
  \begin{align*}
    \lambda_1 = -\frac{\hbar}{r-r'}\qquad
    \mu_1 = -\frac{\hbar'}{r-r'}.
  \end{align*}
  Finally,
  \begin{align*}
    \lambda_0\eps = \frac{\hbar r'}{r-r'}
    =\frac{\hbar\frac{r\hbar'}{\hbar}}{r-\frac{r\hbar'}{\hbar}}
    =\frac{\hbar\hbar'}{\hbar-\hbar'}
  \end{align*}
  giving
  \begin{equation*}
    \nabla_u\nabla_v\xi-\nabla_v\nabla_u\xi=\frac{i}{\lambda_0\eps}\xi
    =i\frac{\hbar-\hbar'}{\hbar\hbar'}\xi.\qedhere
  \end{equation*}

\end{proof}

\noindent
It is noteworthy that the curvature of the bimodule connection only
depends on $\hbar$ and $\hbar'$ and is consequently independent of the
particular choice of parameters that defines the bimodule.  For the
noncommutative torus, one proceeds to define compatible left and right
hermitian structures, which implies first of all that the modules are
projective and secondly, that certain torus algebras for different
values of $\theta$ are Morita equivalent. It would be interesting to
obtain similar results for the noncommutative catenoid.

\section*{Acknowledgement}

\noindent
J.A would like to thank G. Landi for inspiring discussions, and the
University of Trieste for hospitality. This work was partly funded by
the Swedish Research Council and the EU COST action QSPACE.

\bibliographystyle{alpha}
\bibliography{nccat}  

\def\polhk#1{\setbox0=\hbox{#1}{\ooalign{\hidewidth
  \lower1.5ex\hbox{`}\hidewidth\crcr\unhbox0}}}
\begin{thebibliography}{DVMMM96}

\bibitem[AC10]{ac:ncgravitysolutions}
P.~Aschieri and L.~Castellani.
\newblock Noncommutative gravity solutions.
\newblock {\em J. Geom. Phys.}, 60(3):375--393, 2010.

\bibitem[ACH16]{ach:noncommutative.minimal.surfaces}
J.~Arnlind, J.~Choe, and J.~Hoppe.
\newblock Noncommutative minimal surfaces.
\newblock {\em Lett. Math. Phys.}, 106(8):1109--1129, 2016.

\bibitem[AH13]{ah:quantizedMinimal}
J.~Arnlind and J.~Hoppe.
\newblock The world as quantized minimal surfaces.
\newblock {\em Phys. Lett. B}, 723(4-5):397--400, 2013.

\bibitem[Arn14]{a:curvatureGeometric}
J.~Arnlind.
\newblock Curvature and geometric modules of noncommutative spheres and tori.
\newblock {\em J. Math. Phys.}, 55:041705, 2014.

\bibitem[AW17a]{aw:CGB.sphere}
J.~Arnlind and M.~Wilson.
\newblock On the {C}hern-{G}auss-{B}onnet theorem for the noncommutative
  4-sphere.
\newblock {\em J. Geom. Phys.}, 111:126--141, 2017.

\bibitem[AW17b]{aw:curvature.three.sphere}
J.~Arnlind and M.~Wilson.
\newblock {R}iemannian curvature of the noncommutative 3-sphere.
\newblock {\em J. of Noncommut. Geom. (to appear)}, 2017.
\newblock \texttt{arXiv:1505.07330}.

\bibitem[Ber78]{b:diamondlemma}
George~M. Bergman.
\newblock The diamond lemma for ring theory.
\newblock {\em Adv. in Math.}, 29(2):178--218, 1978.

\bibitem[BM11]{bm:starCompatibleConnections}
E.~J. Beggs and S.~Majid.
\newblock {$*$}-compatible connections in noncommutative {R}iemannian geometry.
\newblock {\em J. Geom. Phys.}, 61(1):95--124, 2011.

\bibitem[CM14]{cm:modularCurvature}
A.~Connes and H.~Moscovici.
\newblock Modular curvature for noncommutative two-tori.
\newblock {\em J. Amer. Math. Soc.}, 27(3):639--684, 2014.

\bibitem[Coh95]{c:skewFields}
P.~M. Cohn.
\newblock {\em Skew fields}, volume~57 of {\em Encyclopedia of Mathematics and
  its Applications}.
\newblock Cambridge University Press, Cambridge, 1995.
\newblock Theory of general division rings.

\bibitem[CT11]{ct:gaussBonnet}
A.~Connes and P.~Tretkoff.
\newblock The {G}auss-{B}onnet theorem for the noncommutative two torus.
\newblock In {\em Noncommutative geometry, arithmetic, and related topics},
  pages 141--158. Johns Hopkins Univ. Press, Baltimore, MD, 2011.

\bibitem[DLL15]{dll:sigma.model.solitions}
L.~Dabrowski, G.~Landi, and F.~Luef.
\newblock Sigma-model solitons on noncommutative spaces.
\newblock {\em Lett. Math. Phys.}, 105(12):1663--1688, 2015.

\bibitem[DVM96]{dvm:central.bimodules}
M.~Dubois-Violette and P.~W. Michor.
\newblock Connections on central bimodules in noncommutative differential
  geometry.
\newblock {\em J. Geom. Phys.}, 20(2-3):218--232, 1996.

\bibitem[DVMMM96]{dvmmm:onCurvature}
M.~Dubois-Violette, J.~Madore, T.~Masson, and J.~Mourad.
\newblock On curvature in noncommutative geometry.
\newblock {\em J. Math. Phys.}, 37(8):4089--4102, 1996.

\bibitem[FK12]{fk:gaussBonnet}
F.~Fathizadeh and M.~Khalkhali.
\newblock The {G}auss-{B}onnet theorem for noncommutative two tori with a
  general conformal structure.
\newblock {\em J. Noncommut. Geom.}, 6(3):457--480, 2012.

\bibitem[FK13]{fk:scalarCurvature}
F.~Fathizadeh and M.~Khalkhali.
\newblock Scalar curvature for the noncommutative two torus.
\newblock {\em J. Noncommut. Geom.}, 7(4):1145--1183, 2013.

\bibitem[Lit31]{l:classalgebras}
D.~E. Littlewood.
\newblock On the {C}lassification of {A}lgebras.
\newblock {\em Proc. London Math. Soc.}, S2-35(1):200, 1931.

\bibitem[MR11]{mr:noncommutative.sigma.model}
V.~Mathai and J.~Rosenberg.
\newblock A noncommutative sigma-model.
\newblock {\em J. Noncommut. Geom.}, 5(2):265--294, 2011.

\bibitem[Ros13]{r:leviCivita}
J.~Rosenberg.
\newblock {L}evi-{C}ivita's theorem for noncommutative tori.
\newblock {\em SIGMA}, 9:071, 2013.

\bibitem[Wal14]{w:nuclear.weyl}
S.~Waldmann.
\newblock A nuclear {W}eyl algebra.
\newblock {\em J. Geom. Phys.}, 81:10--46, 2014.

\end{thebibliography}

\end{document}